\documentclass[a4paper,12pt]{article}

\newenvironment{proof}{\noindent {\bf Proof:}}{\hfill $\Box$}

\newtheorem{theorem}{Theorem}
\newtheorem{lemma}{Lemma}

\newtheorem{corollary}{Corollary}

\newtheorem{definition}{Definition}
\newtheorem{assumption}{Assumption}
\newtheorem{remark}{Remark}

\newtheorem{example}{Example}

\usepackage{booktabs}

\usepackage{tikz}

\usepackage{algorithm}
\usepackage[noend]{algpseudocode}
\makeatletter
\def\BState{\State\hskip-\ALG@thistlm}
\makeatother
\usepackage{algorithmicx}

\usepackage{amsmath}
\usepackage{amssymb}
\usepackage{amsfonts}
\usepackage{graphicx}
\usepackage{ dsfont }

\usepackage[normalem]{ulem}
\usepackage{color}

\usepackage{tikz}
\usepackage{lipsum}

\usepackage{url}

\newcommand{\mr}[1]{\mathrm{#1}}

\newcommand{\Rb}{\mathbb{R}}

\newcommand{\Nb}{\mathbb{N}}

\newcommand{\Kf}{\mathbf{K}}

\newcommand{\bs}{\boldsymbol}

\newcommand{\new}[1]{{\color{black}#1}}
\def\be{\begin{equation}}
\def\ee{\end{equation}}

\algrenewcommand\algorithmicensure{\textbf{Output:}}

\usepackage{bbm}
\newcommand{\ones}{\mathbbm{1}}

\title{\bf Stability and performance verification of dynamical systems controlled by neural networks: algorithms and complexity}

\author{Milan Korda$^{1,2}$}

\begin{document}

\maketitle
\thispagestyle{empty}

\footnotetext[1]{Milan Korda is with the CNRS; LAAS; 7 avenue du colonel Roche, F-31400 Toulouse; France ({\tt korda@laas.fr}) as well as with the Faculty of Electrical Engineering, Czech Technical University in Prague,
Technick\'a 2, CZ-16626 Prague, Czech Republic.}
\footnotetext[2]{This work has been supported by the Czech Science Foundation (GACR) under contract No. 20-11626Y, the European Union’s Horizon 2020 research and innovation programme under the Marie Skłodowska-Curie Actions, grant agreement 813211 (POEMA), by the AI Interdisciplinary Institute ANITI funding, through the
French “Investing for the Future PIA3” program under the Grant agreement n$^\circ$ ANR-19-PI3A-0004 as well as by the National Research Foundation, Prime Minister’s Office, Singapore under its Campus for Research Excellence and Technological Enterprise (CREATE) programme.}

\begin{abstract}
This work makes several contributions on stability and performance verification of nonlinear dynamical systems controlled by neural networks.  First, we show that the stability and performance of a polynomial dynamical system controlled by a neural network with semialgebraically representable activation functions (e.g., ReLU) can be certified by convex semidefinite programming. The result is based on the fact that the semialgebraic representation of the activation functions and polynomial dynamics allows one to search for a Lyapunov function using polynomial sum-of-squares methods.  Second,  we remark that even in the case of a linear system controlled by a neural network with ReLU activation functions, the problem of verifying asymptotic stability is undecidable.  Finally,  under additional assumptions,  we establish a converse result on the existence of a polynomial Lyapunov function for this class of systems. \new{Numerical results with code available online on examples of state-space dimension up to 50 and neural networks with several hundred neurons and up to 30 layers demonstrate the method.}
\end{abstract}

\section{Introduction}
The recent wide-spread success and adoption of neural networks in machine learning naturally lead to their applications in safety-critical domains such aerospace or automotive, thereby raising questions of safety. This work addresses this question in the setting of nonlinear dynamical systems controlled by neural network controllers (see Figure~\ref{fig:NN}). We present a method to certify stability of this closed-loop interconnection using convex semidefinite programming (SDP), under the assumption that the dynamics is polynomial and the activation functions in the neural network are semialgebraically representable (e.g., ReLU). Similarly, we derive SDPs that yield bounds on performance in terms of the nonlinear $L_2$ gain or assess robustness and input-to-state stability.

The SDPs provide sufficient conditions of the type ``If a certain SDP is feasible, then the system is stable''. The size of the SDPs can be increased in order to augment their expressive power and hence the chance of finding a stability certificate. On the other hand, if the SDP is not feasible, nothing can be concluded about the stability of the closed-loop interconnection. In fact, we prove a negative complexity result stating that the problem of deciding stability of a \emph{linear} system controlled by a ReLU neural network is \emph{undecidable} in the Turing computational model. This immediately implies the non-existence of a bound on the size of the SDPs required for stability certification computable from the input data. Therefore, there may exist bad instances of linear systems and neural networks for which the size of the SDPs required for stability certification grows to infinity or, possibly, for which all of the SDPs are  infeasible despite the closed-loop interconnection being stable.

On the other hand,  if one  assumes exponential rather than asymptotic stability (and a certain technical assumption holds),  one can prove a converse result on the existence of a polynomial Lyapunov function on compact sets for polynomial systems controlled by neural networks; this result builds on the result of~\cite{peet2009exponentially} for polynomial vector fields without a neural network in the loop.


The method presented here builds on the general framework for verifying stability of semialgebraically representable difference inclusions developed in~\cite{korda2017stability}. In relation to our method, we are aware of the works~\new{\cite{yin2020stability,fazlyab2020safety}} that tackle the problem using integral quadratic constraints, which can be seen as \new{an alternative to the approach presented here where in our case we use an exact representation of the graph of the nonlinearities appearing in the neural network whereas~\cite{yin2020stability,fazlyab2020safety} rely on typically inexact sector inclusions or integral quadratic constraints}. Our result on undecidability is to the best of our knowledge novel although it relies heavily on the work~\cite{blondel2001stability} related saturated systems. The contributions of our work can be summarized as follows:
\begin{itemize}
\item A very general framework for stability and performance analysis, encompassing all polynomial dynamical systems and semialgebraically representable neural networks,  based on convex semidefinite programming.
\item Proof of undecidability of the problem of global stability of polynomial (and even linear) systems controlled by neural networks.
\item A converse Lyapunov result for exponentially stable closed-loop interconnections of polynomial systems and neural networks on compact sets.
\end{itemize}

 

\section{Problem setting}
In this work we consider the closed-loop interconnection of a nonlinear dynamical system and a neural network as depicted in Figure~\ref{fig:NN}.

\begin{figure}[t] 
\centering
\includegraphics[scale=0.6]{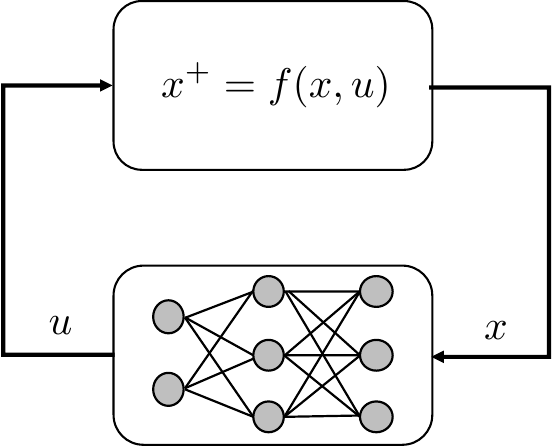} 
\caption{\small Nonlinear dynamical system controlled by a neural network.} \label{fig:NN}
\end{figure}

Specifically, we consider discrete-time dynamical systems of the form
\begin{equation}\label{eq:sys}
x^+ = f(x,u)
\end{equation}
with $x \in \Rb^n$ being the state, $x^+ \in \Rb^n$ the successor state, $u\in \Rb^m$ the control input and $f:\Rb^n\times \Rb^m\to\Rb^n$ a polynomial transition mapping. The goal is the verify the closed-loop stability and performance of system~(\ref{eq:sys}) when controlled by a neural network controller $u = \bs\psi(x)$. That is, the object of interest is the system
\begin{equation}\label{eq:cl}
x^+ = f(x,\bs\psi(x)),
\end{equation}
where $\bs\psi$ is a neural network of the form
\begin{equation}\label{eq:NN}
\bs\psi(x) = W_N(\ldots \rho_2(W_2 \rho_1(W_1x + b_1)+b_2)\ldots   )+b_N
\end{equation}
for some weight matrices $W_i$ and bias vectors $b_i$. The activation functions $\rho_i$, applied componentwise on the output of each layer, are assumed to be semialgebraic; this is satisfied, e.g., for the ReLU, leaky ReLU or the saturation function\footnote{The saturation function is typically applied at the output layer in order to enforce satisfaction of bounds on the control. \new{Other activation functions such as $\tanh$ or sigmoids are not semialgebraic and hence cannot be treated using the presented approach without a further approximation.}}. The semialgebraicity of the activation functions implies that the graph of the function $\bs\psi$ can be expressed as
\begin{align}\nonumber
\mr{graph}_{\bs\psi} = \{(x,u) \mid \exists\, \lambda\in\Rb^{n_\lambda} \;\;\mr{s.t.} \;\; & g(x,u,\lambda) \ge 0,\\
& \; h(x,u,\lambda) = 0  \}\label{eq:graphNN}
\end{align}
for some vectors of polynomials $g$ and $h$ and lifting variables $\lambda$ associated to the semialgebraic functions $\rho_i$ in $\bs\psi$.  We recall that the graph of a function $\bs\psi:\Rb^n\to \Rb^m$ is a subset of $\Rb^{n+m}$ defined as
\[
\mr{graph}_{\bs\psi}  = \big\{\big(x,\bs\psi(x)\big) \mid x\in \Rb^n \big\}.
\]
 Since for each $x$,  the control input $u$ satisfies $u = \bs\psi(x)$, it follows that $(x,u) \in \mr{graph}_{\bs\psi}$ and hence also
\begin{equation}\label{eq:contInclusion}
u \in \Kf_x,
\end{equation}
where the set $\Kf_x$ is given by
\[
\Kf_x = \big\{u \mid  \exists \lambda\in\Rb^{n_\lambda} \;\;\mr{s.t.}\;\;  g(x,u,\lambda) \ge 0,\; h(x,u,\lambda) = 0  \big\}.
\]
We note that in this case, for each $x\in\mathbb{R}^n$, the set $\mathbf{K}_x$ is a singleton although the approach of~\cite{korda2017stability} that this work is based on applies to non-singleton sets $\Kf_x$ as well.

\begin{example}\label{ex:relu}[ReLU] Consider the single-neuron network with a ReLU activation function, i.e.,
\[
\bs\psi(x) = \mr{ReLU}(w^\top x + b) = \max(w^\top x + b,0)
\]
for some vector of weights $w\in \Rb^n$ and a bias $b \in \Rb$. \new{The graph of the function $y = \mr{ReLU}(z) = \max(0,z)$ is given by
\[
\mr{graph}_{\mr{ReLU}} = \big\{(z,y)\mid y \ge z,\; y \ge 0,\; y(y-z) = 0\big\}.
\]
Substituting $w^\top x + b$ for $z$ and $u$ for $y$ },  it follows that the set $\Kf_x$ is given by
\[
\Kf_x = \{u\mid u \ge w^\top x + b,\; u\ge 0,\; u(u - w^\top x - b) = 0 \}.
\]
We note that in this case, no lifting variables $\lambda$ are needed.
\end{example}

\begin{example}[Saturation function] Consider the single-neuron network with the activation function being the saturation at $+1$ and $-1$, i.e.,
\[
\bs\psi(x) = \mr{sat}(w^\top x + b) = \min(\max(w^\top x + b,-1),1)
\]
for some vector of weights $w\in \Rb^n$ and a bias $b \in \Rb$. The graph of the saturation function is given by
\new{
\begin{align}
\mr{graph}_{\mr{sat}} = \big\{(z,y)\mid \exists\,  & \lambda\in\Rb \;\mr{s.t.}\; -1\le y \le 1,\;y \ge z-\lambda,    \nonumber \\& \hspace{-30mm} \lambda\ge 0  , \; (y-1)\lambda = 0, \; (y+1)(y-z+\lambda) =  0 \big\}. 
\end{align}
Substituting $w^\top x + b$ for $z$ and $u$ for $y$, it follows that the set $\Kf_x$ is given by}
\begin{align}
\Kf_x &= \{u \mid \exists\, \lambda\in\Rb \;\mr{s.t.}\; -1\le u \le 1,\; u \ge w^\top x + b-\lambda,   \nonumber \\ 
& \hspace{0mm} \lambda \ge 0, \; (u-1)\lambda = 0, \; (u+1)(u-w^\top x - b+\lambda) =  0 \}. 
\end{align}
In this case, one lifting variable $\lambda$ is needed.
\end{example}

To summarize this section,  the stability of (\ref{eq:cl}) is \emph{equivalent} to the stability of the difference inclusion
\[
x^+ \in f(x,\Kf_x) = \big\{f(x,u) \mid u\in \Kf_x \big\}.
\]
This observation is crucial for stability and performance analysis developed in the subsequent sections.

\section{Stability analysis}\label{sec:stability}
Stability of~(\ref{eq:cl}) can be analyzed using sum-of-squares (SOS) techniques, analogously to~\cite{korda2017stability}, where such analysis was carried out in a more general setting. Before proceeding, we recall a classical definition of stability.
\begin{definition}[Stability]\label{def:stab}
The system~(\ref{eq:cl}) is called globally asymptotically stable if the following two conditions hold:
\begin{enumerate}
\item For all initial conditions $x_0$, it holds $\lim\limits_{k\to\infty} x_k = 0$ (Global attractivity).
\item For all $\epsilon > 0$ there exists $\delta > 0$ such that if $\|x_0\|_2 \le \delta$, then $\|x_k\|_2 \le \epsilon$ for all $k$ (Lyapunov stability).
\end{enumerate}
\end{definition}

We propose to use the following set of Lyapunov conditions to assess the stability of~(\ref{eq:cl}):
\begin{align}
V(x^+,u^+,\lambda^+) - V(x,u,\lambda) &\le -\|x\|^2 \label{eq:Lyap1} \\
V(x,u,\lambda) & \ge 0 \label{eq:Lyap2}
\end{align}
for all
\[
(x,u,\lambda,x^+,u^+,\lambda^+) \in \Kf, 
\]
where
\begin{align*} \Kf =  \big\{(x,u,\lambda,x^+,u^+,\lambda^+) \mid  x^+ = f(x,u),\; & g(x,u,\lambda) \ge 0, \\ & \hspace{-60mm} h(x,u,\lambda) = 0, g(x^+,u^+,\lambda^+) \ge 0,\; h(x^+,u^+,\lambda^+) = 0 \big\}.
\end{align*}

Notice that $V$ is allowed to depend on the output of the neural network $u$ as well as the lifting variables $\lambda$, which increases the richness of the function class we search (allowing, for example, for piecewise polynomial functions of $x$ when $V$ is projected on the $x$ coordinate). Provided that~(\ref{eq:Lyap1}) and (\ref{eq:Lyap2}) hold for a continuous function $V$, the system~(\ref{eq:cl}) is globally attractive and provided that also a mild technical condition on $g$ and $h$ is satisfied, the system (\ref{eq:cl})  is globally asymptotically stable; this is formally proven in the following theorem.

\begin{theorem}\label{thm:main}
If a continuous function $V$ satisfies (\ref{eq:Lyap1}) and (\ref{eq:Lyap2}), then:
\begin{enumerate}
\item The system~(\ref{eq:cl}) is globally attractive, i.e., $x_k \to 0$ for all initial conditions.
\item \new{If in addition the neural network $\bs\psi(x) $ is continuous and there exists a continuous selection for the lifting variables, i.e., there exists a continuous function $\bar\lambda:\Rb^n\to\Rb^{n_\lambda}$ such that for all $x\in \Rb^n$ it holds $g(x,\bs\psi(x),\bar\lambda(x)) \ge 0$ and $h(x,\bs\psi(x),\bar\lambda(x)) = 0$ or $V$ does not depend on $(u,\lambda)$, then the system~(\ref{eq:cl}) is globally asymptotically stable.}

\end{enumerate}  
\end{theorem}
\begin{proof}
Let $(x_k)_{k=0}^\infty$ be a trajectory of~(\ref{eq:cl}) and let $u_k = \bs\psi(x_k)$. By construction of $\Kf_x$, there exists a sequence $(\lambda_k)_{k=0}^\infty$ such that $g(x_k,u_k,\lambda_k) \ge 0$ and $h(x_k,u_k,\lambda_k) = 0$. Since $x_{k+1} = f(x_k,\bs\psi(x_k))$, it follows that
\[
(x_k,u_k,\lambda_k,x_{k+1},u_{k+1},\lambda_{k+1}) \in \Kf
\]
for all $k$. Therefore
\[
V(x_{k+1},u_{k+1},\lambda_{k+1}) - V(x_k,u_k,\lambda_k) \le - \|x_k\|_2^2.
\]
Given $N > 0$ and summing up over $k$ leads to
\[
V(x_{N+1},u_{N+1},\lambda_{N+1}) - V(x_0,u_0,\lambda_0) \le -\sum_{k=0}^N \|x_k\|_2^2.
\]
Therefore for all $N > 0$  \begin{align} \nonumber
\sum_{k=0}^N \|x_k\|_2^2 & \le V(x_0,u_0,\lambda_0) - V(x_{N+1},u_{N+1},\lambda_{N+1}) \\ & \le V(x_0,u_0,\lambda_0) \label{eq:auxProof}
\end{align}
since $V$ is nonnegative. This implies that $x_k \to 0$, proving global attractiveness.

In order to prove Lyapunov stability (condition 2 of Definition~\ref{def:stab}), fix $\epsilon > 0$ and assume for the purpose of contradiction that there exists a sequence of initial conditions $x_0^i \to 0$ and times $k_i$ such that $\|x_{k_i}^i\|_2^2 > \epsilon$, where $x_{k_i}^i$ denotes  the solution to~(\ref{eq:cl}) starting from $x_0^i$ evaluated at time $k_i$. \new{Define $\tilde V(x)  = V(x, \bs\psi(x),\bar \lambda(x))$ with, by assumption, $\bs\psi$ and $\bar \lambda$ continuous and satisfying $g(x, \bs\psi(x),\bar\lambda(x)) \ge 0$ and $h(x,\bs\psi(x),\bar\lambda(x)) = 0$ for all $x \in \Rb^n$. It follows  that $\tilde V$ is continuous and satisfies
$\tilde{V}(x) \ge 0$ and
\[
\tilde{V}(f(x,\bs\psi(x))) - \tilde{V}(x) \le -\|x\|^2
\]
for all $x\in \Rb^n$. By the same calculation as in the first step of the proof, it follows that
\[
\|x_{k_i}\|^2_2 \le \tilde V (x_0^i) - \tilde V(x_{{N+1}}^i) 
\]
for all $i$ and all $N \ge k_i$. Since $x_0^i \to 0$ as $i\to \infty$ and $x_{N+1}^i \to 0$ as $N\to\infty$ (by the first part of the theorem) it follows that for every $\delta >0$  there exist $i_0$ and $N_0$ such that $\| x_0^{i_0} - x_{N_0+1}^{i_0} \|< \delta$. Since $\tilde V$ is continuous it is also locally uniformly continuous and hence this $\delta$ can be chosen small enough such that $\tilde V (x_0^{i_0}) - \tilde V(x_{N_0+1}^{i_0}) < \epsilon$. This implies that $\|x_{k_{i_0}}\| < \epsilon$, which is a contradiction, proving Lyapunov stability. With $V$ independent of $(u,\lambda)$, the same conclusion follows with $\tilde V$ replaced by $V$.}
\end{proof}



\begin{remark}\label{rem:bnd}
The \new{continuous selection} assumption of point~2 of the preceding theorem is satisfied by most commonly used neural networks, including the ReLU network. This is highlighted in the following corollary.
\end{remark}

\begin{corollary}\label{cor:VstabRELU}
\new{If $\bs\psi$ is a neural network with ReLU activation functions modeled as in Example~\ref{ex:relu},  then the neural network $\bs\psi$ is continuous and there exists a continuous function $\bar\lambda$ satisfying $g(x,\bs\psi(x),\bar\lambda(x))\ge 0$ and $h(x,\bs\psi(x),\bar\lambda(x))= 0$ where $g$ and $h$ describe the graph of the neural network~(\ref{eq:graphNN}).
} In addition,  if a continuous function $V$ satisfies (\ref{eq:Lyap1}) and (\ref{eq:Lyap2}) with such a neural network, then the system~(\ref{eq:cl}) is globally asymptotically stable.
\end{corollary}
\begin{proof}
\new{Observe that if modeled as in Example~\ref{ex:relu},  then for each $x$ the polynomial system $h(x,u,\lambda) = 0$ $\&$ $g(x,u,\lambda) \ge 0$ has a unique solution with $u$ being the output of the neural network $\bs\psi(x)$ and $\lambda$ being the outputs of all hidden layers. Defining $\bar\lambda(x)$ to be the outputs of all hidden layers and observing that the ReLU nonlinearity is continuous, the continuity of $\bs\psi$ and $\bar\lambda$ follows. The second claim follows in view of Theorem~\ref{thm:main}.}
\end{proof}

\subsection{Semidefinite programming verification} \label{sec:stabilitySDP}

Since $\Kf$ is basic semialgebraic, a polynomial Lyapunov function $V$ can be searched by replacing the inequality constraints~(\ref{eq:Lyap1}) and (\ref{eq:Lyap2}) by sufficient sum-of-squares conditions. Specifically, denoting $\xi = (x,u,\lambda,x^+,u^+,\lambda^+)$, the inequalities (\ref{eq:Lyap1}), (\ref{eq:Lyap2}) are replaced by

\begin{subequations}\label{eq:sos_eqs}
\begin{align}\label{eq:sos1}
\nonumber V(x,u,\lambda) - V(x^+,u^+,\lambda^+)  - \| x \| ^2_2 = \sigma_0(\xi) \\\nonumber   &\hspace{-6cm} + \sigma_1(\xi)^\top g(x,u,\lambda) + p_1(\xi)^\top h(x,u,\lambda)) \\ \nonumber & \hspace{-6cm} + \sigma_2(\xi)^\top g(x^+,u^+,\lambda^+)   + p_2(\xi)^\top  h(x^+,u^+,\lambda^+)\\ & \hspace{-6cm}  + p_3(\xi)(x^+ - f(x,u))  \\[8pt]  \label{eq:sos2}
&\hspace{-6.4cm} V(x,u,\lambda) = \bar{\sigma}_0 (x,u,\lambda) + \bar{\sigma}_1(x,u,\lambda)^\top g(x,u,\lambda) \\  &  \hspace{-4cm} + \bar p_1(x,u,\lambda)^\top h(x,u,\lambda), \nonumber   
\end{align}
\end{subequations}
where  $\sigma_0,\sigma_1,\sigma_2, \bar \sigma_0,\bar \sigma_1$ are (vectors of) polynomial sum of squares and $p_1$, $p_2$, $\bar p_1$ are (vectors of) polynomials.

From the previous discussion we conclude that stability of~(\ref{eq:cl}) can be assessed by the following SOS feasibility problem:
\begin{equation}\label{opt:sos}
\begin{array}{rclll}
& \mathrm{find} & \displaystyle V,\sigma_0,\sigma_1,\sigma_2,p_1,p_2,p_3,\bar{\sigma}_0,\bar{\sigma}_1,\bar{p}_1 \vspace{1.2mm}\\
& \mathrm{s.t.} & (\ref{eq:sos1}), (\ref{eq:sos2}) \vspace{0.4mm}\\
&& \sigma_0,\sigma_1,\sigma_2, \bar{\sigma}_0,\bar{\sigma}_1 & \hspace{-2.1cm}\text{SOS polynomials}\\
&& V,p_1, p_2,p_3, \bar{p}_1 &\hspace{-2.1cm}\text{arbitrary polynomials},
\end{array}
\end{equation}
where the decision variables are the coefficients of the polynomials \[(V,\sigma_0,\sigma_1,\sigma_2,p_1,p_2,p_3,\bar{\sigma}_0,\bar{\sigma}_1,\bar{p}_1).\] 

The two equality constraints~(\ref{eq:sos1}), (\ref{eq:sos2}) are imposed by comparing coefficients of the polynomials and hence lead to affine constraints on these coefficients. The constraint that a polynomial $\sigma$ of degree $2d$ is sum-of-squares is equivalent to the existence of a symmetric positive semidefinite matrix $W$ of size $\binom{n+d}{d}$ such that $\sigma(x) = \beta(x)^\top W \beta(x)$, where $\beta(x)$ is a basis of the space of polynomials of degree at most $d$ (e.g., the monomial basis). Therefore, when the degree of the polynomials is fixed, the optimization problem~(\ref{opt:sos}) translates to a convex semidefinite programming feasibility problem. More details on sum-of-squares programming can be found in~\cite{lasserreBook,parrilo}.

The result of this section is summarized by the following theorem:
\begin{theorem}\label{thm:mainsdp}
If the sum-of-squares optimization problem~(\ref{opt:sos}) is feasible, then:
\begin{enumerate}
\item The system~(\ref{eq:cl}) is globally attractive, i.e., $x_k \to 0$ for all initial conditions.
\item If in addition the second assumption of Theorem~\ref{thm:main} holds, then the system~(\ref{eq:cl}) is globally asymptotically stable.
\end{enumerate}
\end{theorem}
\begin{proof}
Follows from Theorem~\ref{thm:main} since any $V$ feasible in~(\ref{opt:sos}) is continuous and satisfies (\ref{eq:Lyap1}) and (\ref{eq:Lyap2}).
\end{proof}

The followig corollary follows immediately from Corollary~\ref{cor:VstabRELU}.
\begin{corollary}
If the sum-of-squares optimization problem~(\ref{opt:sos}) is feasible with $\bs\psi$ being a neural network with ReLU activation functions modeled as in Example~\ref{ex:relu}, then the system~(\ref{eq:cl}) is globally asymptotically stable.
\end{corollary}
\begin{proof}
This follows from Corollary~\ref{cor:VstabRELU} and the fact that any $V$ feasible in (\ref{opt:sos}) is continuous and satisfies (\ref{eq:Lyap1}) and (\ref{eq:Lyap2}).
\end{proof}

\subsection{Checking stability is undecidable}
Several natural questions arise as to the possible limitations of the proposed method based on semidefinite programming:
\begin{enumerate}
\item  Does there exist a stable closed-loop interconnection of a polynomial system and a neural network controller for which the optimization problem~(\ref{opt:sos}) is infeasible no matter how high the degree of the polynomials in (\ref{opt:sos})?

 
 \item Does there exist an interesting class of systems and neural networks for which an upper bound on the degree of the polynomials in~(\ref{opt:sos}) necessary for certification of stability of a given system can be computed from the knowledge of the coefficients of the polynomial $f$ and the weights of the neural network?
\end{enumerate} 


The answer to the first question is negative, at least for continuous-time systems, since in this case there exist polynomial dynamical systems for which no polynomial Lyapunov function exists~\cite[Proposition 5.2]{bacciotti2006liapunov}. What is perhaps more surprising is that the answer to the second question is negative even for the very simple class of \emph{linear} systems controlled by ReLU neural networks. This is implied by the following result, stating that the stability verification problem in this case is \emph{undecidable}:

\begin{theorem}\label{thm:undecidable}
\new{The problem of deciding the global asymptotic stability of $x^+ = Ax + B\bs\psi(x)$ is undecidable in the Turing computation model, assuming the input of the decision algorithm is the rational matrices $A \in \mathbb{Q}^{n\times n}$, $B \in \mathbb{Q}^{n\times m}$ and the rational weights and biases of a ReLU neural network.}
\end{theorem}

\begin{proof} \looseness-1 \new{Theorem 2.1 of~\cite{blondel2001stability} proves the undecidability of the global asymptotic stability problem for the saturated systems of the form $x^+ = \mr{sat}(Dx)$, where $\mr{sat} = \min(\max(x,-1),1)$ is the saturation function applied componentwise and $D$ is an $n$-by-$n$ rational matrix}. Using the observation that
\[
\mr{sat}(x) = \mr{ReLU}(x+1) - \mr{ReLU}(x-1) - 1,
\]
one can express any saturated linear system $x^+ = \mr{sat}(Dx)$ in the form of $x^+ = Ax + B\bs\psi(x)$ by taking $A = 0$, $B = \mathbb{I}$ and
\[
W_1 = \begin{bmatrix}
D \\ D
\end{bmatrix},\quad b_1 = \begin{bmatrix} \ones\\ -\ones \end{bmatrix},\quad W_2 = [\mathbb{I}, -\mathbb{I}], \quad b_2 = -\ones,
\]
where $\ones$ is the vector of ones and $\mathbb{I}$ is the identity matrix. \new{Therefore the problem of deciding asymptotic stability of a saturated system can be reduced to the problem of deciding asymptotic stability of a neural network with ReLU activation functions. Since the former is undecidable, so is the latter.}
\end{proof}

\subsection{Converse result for exponential stability}
We finish the section on stability with a positive result.  Namely we show that under a stronger assumption related to exponential stability, a \emph{polynomial} Lyapunov function exists for~(\ref{eq:cl}) on \emph{compact sets}.  This result is analogous to the result of~\cite{peet2009exponentially} for polynomial vector fields. \new{In what follows the symbol $C^k(X)$, $k\in \mathbb{N}\cup\{\infty\}$, denotes the space of $k$-times continuously differentiable functions on a set $X\subset \Rb^n$.}  
\begin{assumption}\label{as:exp}
There exists a function $W \in C^2(\Rb^n)$
\begin{align}
W(f(x,\bs\psi(x))) - W(x) &\le -\|x\|^2_2 \label{eq:Lyap1exp_aux} \\
W(x) & \ge \|x\|^2_2 \label{eq:Lyap2exp_aux}
\end{align}
\end{assumption}
This assumption holds if the quadratic terms in (\ref{eq:Lyap1exp_aux}) and (\ref{eq:Lyap2exp_aux}) are replaced by class-$K$ functions  whenever (\ref{eq:cl}) is asymptotically stable (in this case one can find even a $W \in C^\infty(\Rb^n)$).  Alternatively,  one can construct a $W \in C^\infty(\Rb^n \setminus \{0\}) \cap C(\Rb^n)$ (i.e., smooth away from the origin) satisfying  (\ref{eq:Lyap1exp_aux}) and (\ref{eq:Lyap2exp_aux}) whenever the system~(\ref{eq:cl}) is exponentially stable according to the following definition:

\begin{definition} [Exponential stability]
The system~(\ref{eq:cl}) is called globally exponentially stable if it is Lyapunov stable per Definition~\ref{def:stab} and there exist $\alpha \in [0,1)$ and $C > 0$ such that for all initial conditions $x_0 \in \Rb^n$ and all $k\in \Nb$ we have  $\| x_k\|_2 \le C \alpha^k   \|x_0\|_2 $.
\end{definition}

However,  we did not manage to find or prove a result establishing the existence of $W \in C^2(\Rb^n)$ satisfying (\ref{eq:Lyap1exp_aux}) and (\ref{eq:Lyap2exp_aux}) under the assumption of exponential stability only.

\begin{theorem}\label{eq:thmExp}
Let Assumption~\ref{as:exp} hold,  let $\bs\psi$ be Lipschitz continuous and let $\mathbf{X}$ be a given compact set. Then there exists a polynomial $V$ such that
\begin{align}
V(x^+) - V(x) &\le -\|x\|^2_2 \label{eq:Lyap1exp} \\
V(x) & \ge \|x\|^2_2 \label{eq:Lyap2exp}
\end{align}
for all $(x,x^+) \in \mathbf{K}'$, where 
\begin{align*}
\mathbf{K}' = \{ (x,x^+) \in \Rb^{2n} \mid\; &x \in \mathbf{X},\; \exists \; (u,\lambda)\; \mr{s.t.}\; g(x,u,\lambda) \ge 0, \\ & h(x,u,\lambda) = 0,\; x^+ = f(x,u)\}
\end{align*}
\end{theorem}
\begin{proof}
Let $\bar f (x) = f(x,\bs\psi(x))$ and observe that $\bar f$ is Lipschitz on $\mathbf{X}$ since $f$ is polynomial and $\bs\psi$ Lipschitz. By Assumption~\ref{as:exp}, there exists a $C^2$ function $W:\Rb^n\to \Rb$ such that 
\begin{align*}
W (\bar f(x)) - W(x) \le -\|x\|^2_2 \\
W(x) \ge \|x\|^2_2.
\end{align*}
Since $\mathbf{X}$ is compact and $\bar f$ continuous, $\bar f (\mathbf{X})$ is compact.  Therefore, using Lemma~6 in \cite{peet2009exponentially}, given an $\epsilon > 0$ there exists a polynomial $V$ such that
\[
\sup_{x\in \mathbf{X} \cup \bar f(\mathbf{X})} \frac{|W(x) - V(x)| }{\|x\|_2^2} < \epsilon.
\]
Then we have
\begin{align*}
\frac{V (\bar f(x)) - V(x)}{\|x\|_2^2} & = \frac{W (\bar f(x)) - W(x)}{\|x\|_2^2}   \\ & + \frac{V (\bar f(x)) - W (\bar f(x))}{\|x\|_2^2}+ \frac{W(x) - V(x)}{\|x\|_2^2} \\
& \hspace{-2.5cm} \le \! -1 + \frac{\|\bar f(x)\|_2^2} {\|x\|_2^2}\frac{V (\bar f(x)) - W (\bar f(x))}{\|\bar f(x)\|_2^2} + \epsilon \le \!\! -1 + L^2\epsilon + \epsilon,
\end{align*}
where $L$ is the Lipschitz constant of $\bar f$ (we used the fact that $\bar f(0) = 0$ in the last step). Therefore
\[
V (\bar f(x)) - V(x) \le -[1 - \epsilon(1+L^2)] \cdot \| x\|_2^2.
\]
Similar but simpler argument shows that
\[
V \ge (1-\epsilon)\|x\|_2^2.
\]
Picking $\epsilon < (1+L^2)^{-1}$ implies $V(\bar f(x)) - V(x) \le -\alpha \|x \|_2^2$ and  $V(x) \ge \beta\|x\|_2^2$ for some $\alpha >0$ and $\beta > 0$. Dividing $V$ by $\min(\alpha,\beta)$ yields~(\ref{eq:Lyap1exp}) and (\ref{eq:Lyap2exp}) with $\bar f (x)$ in place of $x^+$.  However, by definition of the set $\mathbf{K}'$, for any $(x,x^+)\in \mathbf{K}'$ it holds $x^+ = f(x,\bs\psi(x)) = \bar f(x)$.
\end{proof}

We remark that the inequalities~(\ref{eq:Lyap1exp}) and (\ref{eq:Lyap2exp}) are a strenghtening of (\ref{eq:Lyap1}) and (\ref{eq:Lyap2}). Specfically, in (\ref{eq:Lyap1exp}) and (\ref{eq:Lyap2exp}), the function $V$ is independent of $(u,\lambda)$ and $V$ is lower bounded by $\|x\|^2_2$ rather than only nonnegative. We stress that the existence of a \emph{polynomial} $V$ satisfying (\ref{eq:Lyap1exp}) and (\ref{eq:Lyap2exp}) is guaranteed only on compact sets. We also remark that $\bs\psi$ is Lipschitz for the ReLU neural network as required by the assumptions of Theorem~\ref{eq:thmExp}.

\section{Performance and robustness analysis}\label{sec:perf_robust}
In this section we briefly outline how the proposed approach extends  to performance and robustness certification. We consider the system of the form
\begin{subequations}\label{eq:sysw}
\begin{align}
x^+ & = f(x,\bs\psi(x),w) \label{eq:sysw1} \\ 
y &= f_y(x)  \label{eq:sysw2},
\end{align}
\end{subequations}
where $y$ is the so-called performance output and $w$ is the disturbance taking values in the possibly state and control dependent set
\[
\mathbf{W}(x,u) = \{w \mid \new{r(x,u,w) \ge 0} \}
\]
\new{with $r$ being a vector of polynomials.} The following set will take place of the $\Kf$ in this setting:
\begin{align*}
 \Kf_w  =  \big\{ & \new{(x,u,\lambda,w,x^+,u^+,\lambda^+)} \mid   x^+ = f(x,u,w),\; \\ &  g(x,u,\lambda) \ge 0, \;  h(x,u,\lambda) = 0, 																										 \new{r(x,u,w)\ge 0}, \; \\ & g(x^+,u^+,\lambda^+) \ge 0,\; h(x^+,u^+,\lambda^+) = 0 \big\}.
\end{align*}

The performance metric chosen is the $\ell_2$ gain from $w$ to $y$; we also treat the closely related robust stabilization and input to state stability. Other performance metrics, both in deterministic and stochastic settings, can be considered using the same computation framework; see \cite[Section 5.3]{korda2017stability}.

\subsection{Nonlinear $\ell_2$ gain} 
We consider the nonlinear $\ell_2$ gain starting from a given initial condition (taken without loss of generality to be zero) defined as
\begin{equation}\label{eq:l2gain}
\inf \Big\{ \alpha     \mid   \sum_{k=0}^\infty \|y_k\|^2_2 \le \alpha^2 \sum_{k=0}^\infty \|w_k\|^2_2,\; x_0 = 0 \Big\},
\end{equation}
where $(y_k)_{k=0}^\infty$ is the output of system~(\ref{eq:sysw}) with zero initial condition and disturbance $(w_k)_{k=0}^\infty$.

An upper bound on the $\ell_2$ gain is provided by the following set of constraints:
 \begin{align}
V(x^+,u^+,\lambda^+,w^+) - V(x,u,\lambda,w) &\le -\|f_y(x)\|^2_2 + \gamma \|w\|_2^2  \ \label{eq:l21} \\
V(x,u,\lambda,w) & \ge 0 \label{eq:l22} \\
V(0,u,\lambda,w) & =0 \label{eq:l23}
\end{align}
for all
\[
(x,u,\lambda,w,x^+,u^+,\lambda^+,w^+) \in \Kf_w.
\]
\begin{lemma}
If $(V,\gamma)$ satisfies (\ref{eq:l21})-(\ref{eq:l22}) for all $(x,u,\lambda,w,x^+,u^+,\lambda^+,w^+) \in \Kf_w$, then the $\ell_2$ gain (\ref{eq:l2gain}) is bounded by $\sqrt{\gamma}$.
\end{lemma}
\begin{proof}
This follows from Lemma 5 and Corollary 1 in~\cite{korda2017stability}.
\end{proof}

In order to find a bound on the $\ell_2$ gain computationally, one solves the optimization problem of minimizing $\gamma$ subject to the constraints (\ref{eq:l21})-(\ref{eq:l22}) enforced via sufficient sum-of-squares constraints as in Section~\ref{sec:stabilitySDP}, leading to a convex SDP.

\subsection{Robust stabilization and Input-to-state stability}
A minor modification of inequalities (\ref{eq:l21})-(\ref{eq:l22}) allows us to verify robust stabilization and input to state stability. In this case, we enforce,
 \begin{align}
V(x^+,u^+,\lambda^+,w^+) - V(x,u,\lambda,w) &\le -\|x\|^2_2 + \gamma \|w\|_2^2  \ \label{eq:iss1} \\
V(x,u,\lambda,w) & \ge \|x\|_2^2 \label{eq:iss2} 
\end{align}
for all
\[
(x,u,\lambda,w,x^+,u^+,\lambda^+,w^+) \in \Kf_w.
\]

The following result follows by combining the arguments of Theorem~\ref{thm:main} in this work and Lemma 5 and Corollary 1 in~\cite{korda2017stability}.
\begin{lemma}
If $(V,\gamma)$ satisfies (\ref{eq:iss1})-(\ref{eq:iss2}) for all $(x,u,\lambda,w,x^+,u^+,\lambda^+,w^+) \in \Kf_w$ with $\gamma < \infty$ and if $\bs\psi$ is the ReLU neural network modeled as in Example~\ref{ex:relu}, then the the system~(\ref{eq:sysw1}) is input-to-state stable (ISS). If these inequalities are satisfied with $\gamma = 0$, then the system~(\ref{eq:sysw1}) is robustly globally asymptotically stable.
\end{lemma}
As before, replacing the inequalities by sufficient sum-of-squares constraints and minimizing $\gamma$, leads to a convex SDP.

\begin{remark}[Local results]
All results on stability, performance and robustness can be localized to a given basic semialgebraic domain of interest $\mathbf{X}$ by adding the polymomial constraints defining $\mathbf{X}$ to the descriptions of the sets $\mathbf{K}$ respectively $\mathbf{K}_w$ in Sections~\ref{sec:stability} respectively~\ref{sec:perf_robust}.
\end{remark}


\section{Numerical results}
In this section we briefly demonstrate the proposed method.  The goal is a proof-of-concept demonstration,  in terms of a computational viability of the method for systems of practically interesting dimensions,  rather than extensive evaluation on close-to-practice examples.  For this purpose,  we consider the problem of stability verification of a linear system  $x^+ = Ax + Bu$ controlled by a neural network.  We consider the neural network \new{of the form~(\ref{eq:NN})
with $\rho$ being the ReLU nonliearity applied componentwise and $b_i = 0$}.  We generate stable interconnections of this form in the following way.  First we generate a random matrix $A \in  \Rb^{n\times n} $ with its entries being independent standard gaussians.  Second,  we scale the matrix such that its spectral norm is strictly less than one.  The matrix $\new{B\in \Rb^{n\times n}}$ is a random matrix of zeros and ones scaled such that its spectral norm is one.  \new{The matrices $W_i$} are generated randomly in the same way as $A$ and scaled such that their spectral radii are equal to one.  Finally,  we make a random (non-unitary) coordinate transformation,  rendering the spectral norm of $A$ strictly greater than one while preserving the stability of the closed-loop interconnection. The stability verification was carried out by solving (\ref{opt:sos}) with a quadratic $V$ and with the degree of the polynomial multipliers $\sigma$ and $p$ chosen such that the degree of all polynomials appearing in~(\ref{eq:sos_eqs}) is at most quadratic. \new{First, we investigated the performance on a neural network with a single hidden layer (i.e., $N=2$ in~(\ref{eq:NN})) }. Table~\ref{tab:timings} reports the time\footnote{The time reported is the time spent by the interior point solver MOSEK running on Matlab and a Macbook Air with 1.2 GHz Quad-Core Intel i7 and  16GB RAM. The time does not include the parsing time of Yalmip~\cite{yalmip} that in some cases dominates the time total computation time. The parsing time could be dramatically reduced by a custom implementation (e.g. in C++ or Julia). } to solve the semidefinite program~(\ref{opt:sos}) for different dimensions of the state-space $n$ and for different numbers of neurons (= numbers of rows of $W_1$ in~(\ref{eq:NN})). Figure~\ref{fig:Lyap} shows one trajectory of the closed loop system and the computed Lyapunov function evaluated along the trajectory. \new{Second, we investigate the behavior when we increase the number of hidden layers while fixing the state-space dimension and the number of neurons per layer; the results are in Table~\ref{tab:multiLayer}.  } The scalability of the approach could be further improved by considering sparsity or symmetries,  with both topics recently developed in the context of sum-of-squares methods for dynamical systems (see~\cite{schlosser2020sparse} for sparsity and \cite {fantuzzi2020symmetries} for symmetries).  The code for the numerical examples is available from

\begin{center}
 \url{https://homepages.laas.fr/mkorda/NN.zip}
\end{center}

\new{Remarkably, for all instances of stable systems randomly generated using the procedure described above, the SDP~(\ref{opt:sos}) was feasible, thereby certifying stability (up to the floating point error of the SDP solver). This cannot be explained (but does not contradict) the theory developed in this paper and suggests that the worst-case complexity used here in Theorem~\ref{thm:undecidable} may not be best suited for this class of problems and one should perhaps try  to analyze the complexity within an alternative framework (e.g., smoothed analysis or averaged analysis). Whenever the generated system was unstable, the SDP~(\ref{opt:sos}) was infeasible, in accordance with the theory.  } 


An interesting direction of future research is to incorporate the approach within an automatic differentiation scheme~\cite{agrawal2019differentiable} in order to be able to tune the weights of the neural network while guaranteeing stability or optimizing performance.

\begin{figure*}[h]
\begin{picture}(140,175)
\put(5,0){\includegraphics[width=75mm]{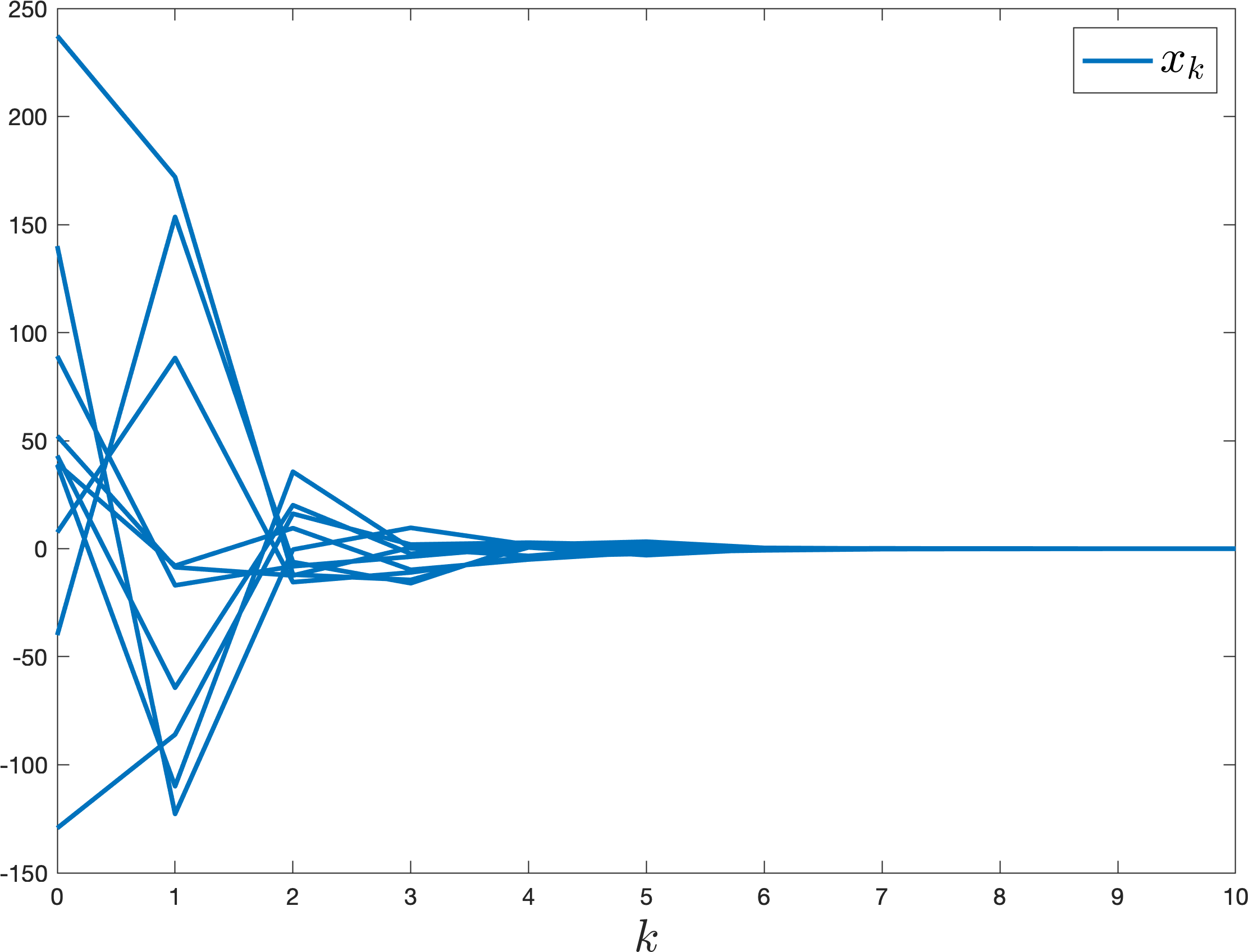}}
\put(245,0){\includegraphics[width=75mm]{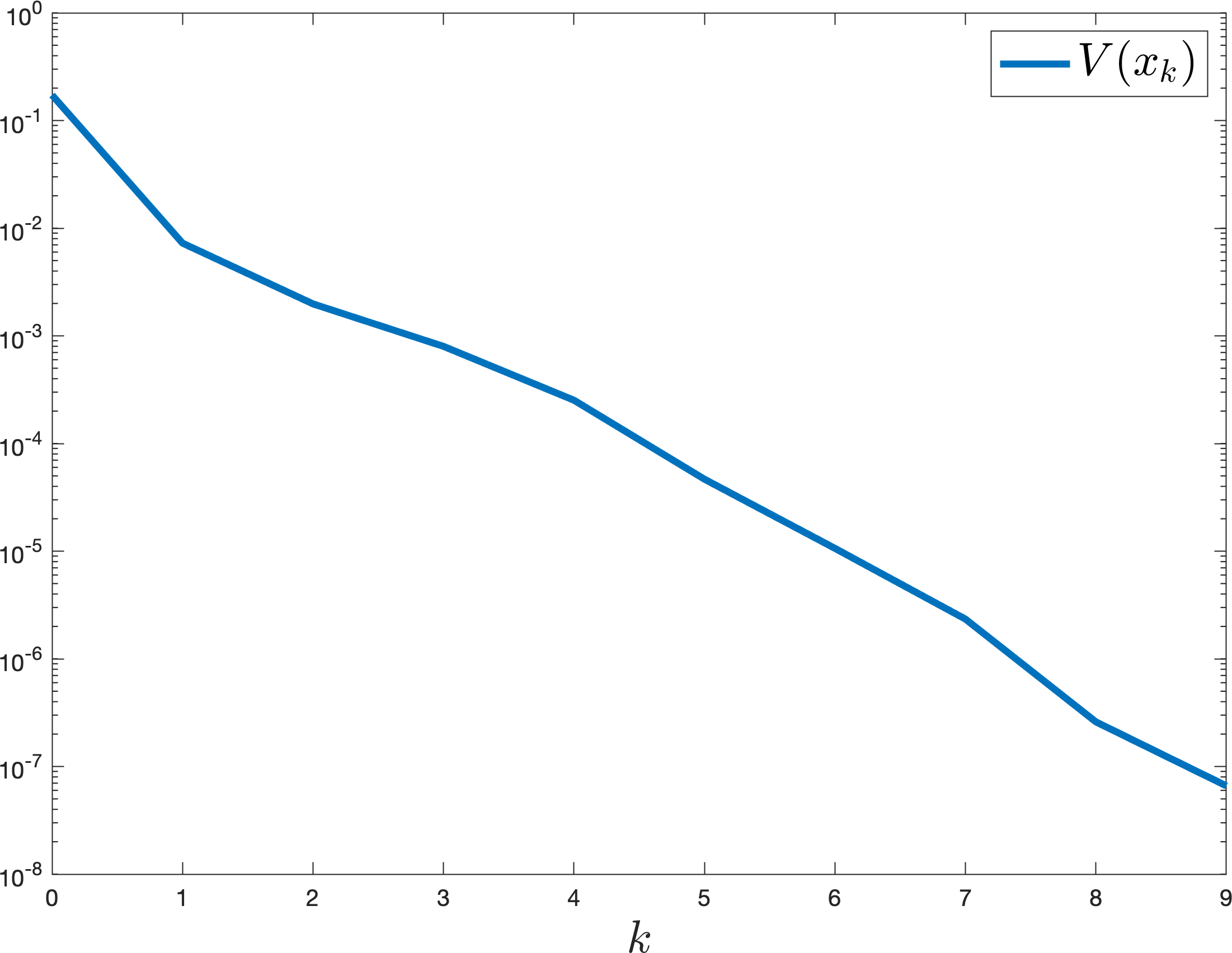}}
\end{picture}
\caption{\small Trajectory of a ten dimensional linear system controlled by a neural network with one hundred neurons (left). The associated Lyapunov function from~(\ref{opt:sos}) evaluated along the trajectory (right).}
\label{fig:Lyap}
\end{figure*}


\begin{table}[h]
\small 
\centering
\caption{\rm \scriptsize  Computation time to solve the SDP~(\ref{opt:sos}) with a single hidden layer neural network for different values of the state-space dimension $n$ and different number of neurons.  The symbol $\times$ signifies that the Yalmip parse time was taking more than two hours.  }\label{tab:timings}
\begin{tabular}{cccccc}
\toprule 
$\#$ Neurons  & 10  & 100  & 200 & 500  \\ \midrule
$n=10$ & 0.61\,s & 1.20\,s & 3.85\,s & 59.04\,s \\\midrule
$n=20$ & 0.94\,s & 10.95\,s  &31.77\,s  & $\times$\\ \midrule
$n=30$ & 1.38\,s & 15.41\,s  & 59.42\,s &$\times$ \\ \midrule
$n=50$ & 45.77\,s & $\times$ &$\times$ &$\times$ \\
\bottomrule
\end{tabular}
\end{table}

\begin{table}[t]
\small 
\centering
\caption{\rm \scriptsize  \new{Computation time to solve the SDP~(\ref{opt:sos}) for different number of Layers of the neural network with 20 neurons per layer and state-space dimension~10.}}\label{tab:multiLayer}
\begin{tabular}{cccccccc}
\toprule 
\new{$\#$ Layers}  & \new{1}  & \new{2}  & \new{5} & \new{10} & \new{20} & \new{30}  \\ \midrule
\new{Time} & \new{0.42\,s} & \new{0.51\,s} & \new{1.02\,s} & \new{1.32\,s} & \new{6.48\,s} & \new{16.2\,s} \\
\bottomrule
\end{tabular}
\end{table}


\section{Acknowledgements}

This work benefited from discussions regarding Assumption~\ref{as:exp} with Luca Zaccarian,  Aneel Tanwani and Andy Teel.


%
%

\bibliographystyle{abbrv}
\bibliography{./References}

\end{document}